\documentclass[12pt]{amsart}
\usepackage{amscd,amssymb,graphics}

\usepackage{amsfonts}
\usepackage{amsmath}
\usepackage{amsxtra}
\usepackage{latexsym}
\usepackage[mathcal]{eucal}

\usepackage{graphics,colortbl,mathrsfs}

\input xy
\xyoption{all}
\usepackage{epsfig}

\usepackage[pdftex,bookmarks,colorlinks,breaklinks]{hyperref}
\makeatletter
\@namedef{subjclassname@2020}{%
  \textup{2020} Mathematics Subject Classification}
\makeatother
\newtheorem{theorem}{Theorem}[section]
\newtheorem{fact}[theorem]{Fact}
\newtheorem{corollary}[theorem]{Corollary}
\newtheorem{lemma}[theorem]{Lemma}

\newtheorem{question}[theorem]{Question}

\numberwithin{equation}{section}

\newtheorem{claim}{Claim}

\newtheorem{case}{Case}
\theoremstyle{definition}
\newtheorem{remark}[theorem]{Remark}

\newtheorem{definition}[theorem]{Definition}

\newcommand{\ben}{\begin{enumerate}}
	\newcommand{\een}{\end{enumerate}}
\newcommand{\bit}{\begin{itemize}}
	\newcommand{\eit}{\end{itemize}}

\def\QED{\nobreak\quad\ifmmode\roman{Q.E.D.}\else{\rm Q.E.D.}\fi}

\def\pfi{\operatorname{pfi}}

	\title[Constructing Psuedo-$\tau$-fine Groups]{Constructing Psuedo-$\tau$-fine Precompact Groups}

\author[D. Peng, G. Zhang]{Dekui Peng, Gao Zhang*}
   \address[D. Peng]
	{\hfill\break Institute of Mathematics,
		\hfill\break Nanjing Normal University, 210023,
	\hfill\break China}
\email{pengdk10@lzu.edu.cn}

	\address[G. Zhang]
	{\hfill\break School of Mathematical Sciences,
		\hfill\break Jiangsu Second Normal University, 211222,
		\hfill\break China}
	\email{gaozhang0810@hotmail.com}
\thanks{*Corresponding Author}

	
\begin{document}
	\maketitle	
\thanks{\begin{center}Dedicated to the 70th anniversary of Professor Mikhail Tkachenko\end{center}}
\begin{abstract}
Let $\tau$ be an uncountable cardinal. The notion of a \emph{$\tau$-fine} topological group was introduced in 2021. More recently, H. Zhang et al. generalized this concept by defining pseudo-$\tau$-fine topological groups to study certain factorization properties of continuous functions on topological groups. It is known that $\tau$-fineness cannot coexist with precompactness in topological groups with uncountable character. In this paper, we investigate this problem further. We prove that, in topological groups with uncountable pseudocharacter, precompactness can coexist with pseudo-$\tau$-fineness for some bounded $\tau$ but pseudocompactness can never.
\end{abstract}
%\tableofcontents

\section{Introduction}
All topological groups in this paper are assumed to be Hausdorff. 

The space of continuous real-valued functions on a topological group plays a key role not only in the theory of topological groups but also in analysis and measure theory. It was perhaps Pontryagin who first observed that every continuous function on a compact group factorizes through a continuous homomorphism onto a separable metrizable topological group.
To say that a function $ f: G \to \mathbb{R} $ factorizes through a continuous homomorphism $ \pi: G \to H $ means that there exists a continuous function $ f' $ on $ H $ such that $ f = f' \circ \pi $.

Inspired by this observation, Tkachenko introduced the notion of $ \mathbb{R} $-factorizable topological groups, which brought us closer to understanding the factorization theory of continuous functions on topological groups. A topological group $ G $ is called $ \mathbb{R} $-factorizable if every continuous function on $ G $ factorizes through a continuous homomorphism $ G \to H $ onto a separable metrizable group. For more details on $ \mathbb{R} $-factorizable topological groups, we refer the reader to \cite[Chapter 8]{AT}.

In \cite{ZPH}, a weaker version of $ \mathbb{R} $-factorizability, called \emph{$\mathscr{M}$-factorizability}, was introduced by the first listed author and his collaborators. This version is obtained by removing the separability condition on $ H $ in the definition of $ \mathbb{R} $-factorizable topological groups.

In the recent work \cite{HPTZ}, the authors introduced the concept of \emph{$\tau$-fine} topological groups\footnote{Although this article is a collaborative effort, it was Tkachenko himself who independently and insightfully defined $\tau$-fine topological groups and discovered the relationship between $\tau$-fineness and $\mathscr{M}$-factorizability.}, as defined in Definition \ref{taufine} below, where $ \tau $ is an uncountable cardinal. It has been shown that $ \tau $-fineness is deeply connected to $ \mathscr{M} $-factorizability. Building on this, Zhang et al. \cite{ZX} introduced the notion of \emph{pseudo-$ \tau $-fine} topological groups to study a weaker version of the factorization theorem.

\begin{definition}\label{taufine}
Let $\tau$ be an infinite cardinal. A topological group $G$ is called 
\begin{itemize} 
\item \emph{$\tau$-fine} if for every family of identity neighbourhoods $\{U_\alpha: \alpha\in\tau\}$, there exists a countable family $\{V_n:n\in \omega\}$ of identity neighbourhoods such that each $U_\alpha$ contains some $V_n$;
\item \emph{pseudo-$\tau$-fine} if for every family of identity neighbourhoods $\{U_\alpha: \alpha\in\tau\}$, there exists a countable family $\{V_n:n\in \omega\}$ of identity neighbourhoods such that $\bigcap_{\alpha\in\tau}U_\alpha\supseteq \bigcap_{n\in\omega}V_n$.
\end{itemize}
\end{definition}

From the definition, if a topological group is of countable character (resp. pseudocharacter), then it is automatically $\tau$-fine (resp. pseudo-$\tau$-fine) for any uncountable cardinal $\tau$.

In \cite{HPTZ}, it was proven that for any uncountable cardinal \(\tau\), every \(\tau\)-fine precompact topological group is metrizable. In this paper, we address the question of whether precompactness and pseudo-\(\tau\)-fineness can coexist in a topological group with uncountable pseudocharacter. We show that for every uncountable cardinal \(\tau \leq \mathfrak{c}^+\), there exists a precompact abelian group \(G\) of uncountable pseudocharacter such that \(\tau\) is the smallest cardinal for which \(G\) is not pseudo-\(\tau\)-fine. As a corollary, we conclude that there exists a pseudo-\(\mathfrak{c}\)-fine precompact abelian group with uncountable pseudocharacter. Additionally, we demonstrate that precompactness cannot be strengthened to pseudocompactness: for any uncountable \(\tau\), pseudocompact pseudo-\(\tau\)-fine topological groups must have countable pseudocharacter (and then it becomes metrizable and compact).

\subsection{Elementary Facts, Notation and Terminology}
We denote $\omega$ and $\omega_1$ as the first infinite and uncountable cardinal numbers, respectively; while the cardinality of the continuum, $2^\omega$, will be noted by $\mathfrak{c}$.
If $\tau$ is an infinite cardinal, we write $\tau^+$ to be its successor.
For a set $A$, $|A|$ is the cardinality of $A$. For a topological space $X$, we will let $w(X)$ to be its {\em weight}, the least cardinality of a base of $X$.
The {\em character} or {\em local weight} of a point $x$ in $X$, which is denoted by $\chi(x)$ (or $\chi(x,X)$), is the least cardinality of a base of $X$ at the point $x$.
If the space $X$ is homogeneous, in particular, if $X$ is a topological group, then all points have the same character; we then call this cardinal the character of $X$ and denote it by $\chi(X)$.
A similar notion is the so-called {\em pseudocharacter}.
The pseudocharacter $\psi(x)$ of $x$ in $X$ is the least cardinal $\tau$ such that there exists a family of open neighbourhoods of $x$ whose intersection is the singleton $\{x\}$.
When $X$ is homogeneous, $\psi(X)$ is defined by a similar way.
For any infinite cardinal $\tau$, a $G_\tau$ set is the intersection of a family of open sets of cardinality at most $\tau$; while $G_\omega$ sets are known as $G_\delta$ sets.
So for a topological group $G$, $\psi(G)\leq \tau$ if and only if $\{1_G\}$ a $G_\tau$ set.
The {\em density} or {\em density character} is the least cardinal of a dense subset of $X$; this cardinal will be noted by $d(X)$ in this paper.

Using the above notion, pseudo-$\tau$-fine topological groups can be defined as the groups $G$ such that every $G_\tau$ set $A\ni 1_G$ contains a $G_\delta$ set $B\ni 1_G$.
It is well-known that any open neighbourhood $U$ of $1_G$ in a topological group $G$ contains a closed subgroup $N$ of $G$ such that the left coset space $G/N$ is of countable pseudocharacter.
So a $G_\tau$ set containing $1_G$, as an intersection of at most $\tau$ many open neighbourhood of $1_G$, always contains a closed subgroup $N$ such that $\psi(G/N)\leq \tau$.
Then we have:
\begin{fact}\label{Fact1}
Let $\tau$ be an uncountable cardinal. A topological group $G$ is pseudo-$\tau$-fine if and only if every closed subgroup $M$ of $G$ satisfying $\psi(G/M) \leq \tau$ contains a closed subgroup $N$ such that $\psi(G/N) \leq \omega$.
\end{fact}

A topological group is called {\em precompact} if its Weil completion is a compact group. 
The following conclusion about precompact groups will be frequently used throughout this paper and will no longer be cited specifically.

\begin{fact}\cite{AT}
An infinite precompact group $G$ has the same character and weight, i.e., $\chi(G)=w(G)$.
\end{fact}

The category $\mathbf{PCA}$ of precompact abelian groups (with continuous homomorphisms) satisfies a nice duality theorem, firstly discovered by Comfort and Ross \cite{CR}.
For any precompact abelian group $G$, we denote by $G^*$ the group of all continuous homomorphisms $G\to \mathbb{T}$ with the pointwise convergence topology, where $\mathbb{T}$ is the circle group, i.e., the 1-dimensional torus.

\begin{fact}\label{Fact2}\cite{Peng1,Peng2}
Let $G$ be a precompact abelian group. Then:
\begin{itemize}
    \item[(1)] $G^*$ is dense in $\widehat{G_d}$, where $G_d$ is the group $G$ endowed with the discrete topology and $\widehat{G_d}$ is the Pontryagin dual group of $G_d$.
    \item[(2)] $(G^*)^* \cong G$.
    \item[(3)] For every closed subgroup $H$ of $G$, the subgroup $H^\perp := \{f \in G^* : f(H) = \{0\}\}$ is closed in $G^*$, and $H^* \cong G^* / H^\perp$.
    \item[(4)] $|G^*| = w(G)$.
    \item[(5)] $d(G^*) = \psi(G)$.
\end{itemize}
\end{fact}
In (2), the topological isomorphism $G\to (G^*)^*$ is given by the {\em evaluation}, i.e., $g\mapsto g^*$, where $g^*(\varphi)=\varphi(g)$ for any $\varphi\in G^*$.
We identify $G$ with $(G^*)^*$ throughout this discussion. Then, $(H^\perp)^\perp = H$ for every closed subgroup $H$ of $G$.

Following Hewitt, we say that a topological space $X$ is {\em pseudocompact} if every continuous real-valued function on $X$ has a bounded range, that is, $C(X) = C^*(X)$. In the influential paper \cite{CR2} by Comfort and Ross, many properties of pseudocompact groups were established. Among other results, they proved the following:

\begin{fact}\label{CR2}
A pseudocompact group is precompact and $G_\delta$-dense in its Weil completion.
\end{fact}

Here, we say that a subspace $Y$ of $X$ is $G_\delta$-dense in $X$ if $Y$ intersects every nonempty $G_\delta$ set of $X$.

By this result, the following is immediate.
\begin{fact}\label{pseudo}
Let $G$ be a pseudocompact group, $H$ a compact metrizable group and $f: G\to H$ a continuous group homomorphism.
Then $f$ is open onto its image and $f(G)$ is closed in $H$.
\end{fact}

\begin{corollary}\label{pseumetr}
A pseudocompact group $G$ with $\psi(G)\leq \omega$ is metrizable and compact.
\end{corollary}

\begin{proof}
By Fact \ref{CR2}, the Weil completion $K$ of $G$ is a compact group and $G$ is $G_\delta$ dense in $K$.
Since $\psi(G)\leq \omega$, there exists a $G_\delta$ set $P$ of $K$ such that $P\cap G=\{1_G\}$.
By $G_\delta$ denseness of $G$ in $K$, the singleton $\{1_G\}$ is dense in $P$.
So $P=\{1_G\}$ and we obtain that $\psi(K)\leq \omega$.
As a compact group, $K$ have the same pseudocharacter and character.
So $K$ is first-countable and hence metrizable.
Now by Fact \ref{pseudo}, we have that the inclusion $G\to K$ is surjective, that is, $G=K$.
\end{proof}

\section{Main Results}

\subsection{Constructing Precompact Abelian Groups with Small Pseudo-fineness Index}
In this subsection, we only discuss abelian groups. So the group operations are written additively.
\begin{lemma}\label{Le:Dual}
Let $G$ be a precompact abelian group and $\tau$ an infinite cardinal. Then $G$ is pseudo-$\tau$-fine if and only if every closed subgroup of $G^*$ with density $\leq \tau$ is covered by a separable subgroup of $G^*$.
\end{lemma}

\begin{proof}
First, assume that $G$ is pseudo-$\tau$-fine. Take a closed subgroup of $G^*$ of density $\leq \tau$. Without loss of generality, we may assume the group is of the form $N^\perp$, where $N$ is a closed subgroup of $G$. Then we have $(G / N)^* \cong N^\perp$ with $\psi(G / N) \leq \tau$.

By the pseudo-$\tau$-fineness of $G$, there exists a closed subgroup $M \leq N$ such that $\psi(G / M) \leq \omega$. Thus, $N^\perp \leq M^\perp$. It remains to note that $M^\perp$ is separable, which follows immediately from $d(M^\perp) = \psi((M^\perp)^*) = \psi(G / M)$.

For the converse, assume that every closed subgroup of $G^*$ with density $\leq \tau$ is covered by a separable subgroup. Let $N$ be a closed $G_\tau$ subgroup of $G$ with $\psi(G/N)\leq \tau$. By assumption, $N^\perp$ is covered by a separable subgroup $M^\perp \leq G^*$, where $M$ is a closed subgroup of $G$. Then  $M\leq N$ and $M$ satisfies $\psi(G / M) \leq \omega$, it follows that $G$ is pseudo-$\tau$-fine.
\end{proof}

It is evident a topological group is of countable pseudocharacter if and only if it is pseudo-$\tau$-fine for all uncountable $\tau$.
Then, for a topological group $G$ with uncountable pseudocharacter, we write $\pfi(G)$ to be the least cardinal $\tau$ such that $G$ is not pseudo-$\tau$-fine; we call this cardinal the {\em pseudo-fineness index} of $G$.
Then $\pfi(G)$ is regular and $\pfi(G)\leq \psi(G)$ \cite{ZX}.

\begin{corollary}\label{Coro}
Let $G$ be a precompact abelian group such that $\psi(G)$ is uncountable. Then 
\begin{itemize}
\item[(i)] either $\pfi(G)\leq \mathfrak{c}$ or $\pfi(G)$ equals to the least cardinal  $\tau>\mathfrak{c}$ such that $G$ has a quotient group of pseudocharacter $\tau$;
\item[(ii)] $\pfi(G)\leq (2^\mathfrak{c})^+$, the successor of $2^\mathfrak{c}$.
\end{itemize}
\end{corollary}

\begin{proof}

(i) If $G$ is pseudo-$\kappa$-fine for every $\kappa\leq \mathfrak{c}$ and $\tau$ is the least cardinal satisfying that $\tau>\mathfrak{c}$ and that $G$ admits a quotient group of pseudocharacter $\tau$.
Then for every $\lambda<\tau$ and every closed subgroup $M$ of $G$ with $\psi(G/M)\leq \lambda$, one has that there exists $\kappa\leq \mathfrak{c}$ such that $\psi(G/M)=\kappa$ by the choice of $\tau$.
Therefore, the pseudo-$\kappa$-fineness of $G$ allows us to find a closed subgroup $M_1$ of $G$ so that $G/M_1$ is of countable pseudocharacter.
This is saying that $G$ is also pseudo-$\lambda$-fine.
We are going to see that $G$ is not pseudo-$\tau$-fine.
By the choice of $\tau$, $G$ admits a closed subgroup $N$ such that $\psi(G/N)=\tau$.
Then, $N^\perp$ is a closed subgroup of $G^*$ of density $\tau$, by Fact \ref{Fact2}.
Since $\tau>\mathfrak{c}$, $N^\perp$ cannot be covered by a separable subgroup of $G^*$, because regular separable spaces are of weight $\leq \mathfrak{c}$ and their subspaces are therefore of density $\leq \mathfrak{c}$.
Now apply Lemma \ref{Le:Dual}.

(ii) Since $\pfi(G)\leq \psi(G)$, we may assume that $\psi(G)>(2^\mathfrak{c})^+$.
So $d(G^*)>(2^\mathfrak{c})^+$ and $G^*$ has a subgroup $A$ with size $(2^\mathfrak{c})^+$. Let $B$ be the closure of $A$.
Then $A$, hence also $B$, cannot be covered by a separable subgroup of $G^*$, because Hausdorff separable spaces are of cardinality $\leq 2^\mathfrak{c}$ \cite[Theorem 1.5.3]{Eng}.
Let $\tau=d(B)$. It follows from Lemma \ref{Le:Dual} that $G$ is not pseudo-$\tau$-fine.
So $\pfi(G)\leq \tau=d(B)\leq |A|=(2^\mathfrak{c})^+$.
\end{proof}

\begin{theorem}\label{pfi}
For any uncountable regular cardinal $\tau\leq \mathfrak{c}^+$, there exists a precompact abelian group $G$ such that $\pfi(G)=\tau$.
\end{theorem}

\begin{proof}
Let $X$ be a set of size $\tau$ and $F(X)$ the free abelian group over $X$.
Then for every subset $Y$ of $X$, the free abelian group $F(Y)$ is naturally considered as a subgroup of $F(X)$, hence $F(X)=F(Y)\oplus F(X\setminus Y)$.
For each $Y\in [X]^{<\tau}$, we let $f_Y: F(X)\to \mathbb{T}$ be a homomorphism with kernel $F(X\setminus Y)$; the existence of such a homomorphism follows from that $F(Y)=\omega\cdot |Y|\leq \mathfrak{c}$.
Let $G=F(X)$ be endowed with the coarsest topology making all $f_Y$ continuous.
It is easy to see that $G$ becomes a Hausdorff precompact group.
So, $G^*$ is a dense subgroup of $\widehat{F(X)}$.
That $G$ is a desired group will follows easily from the next Claim.

\begin{claim}\label{Claim1}For every subset $A$ of $G^*$ with $|A|<\tau$, there exists $Y\in [X]^{<\tau}$ such that $F(X\setminus Y)\subseteq \bigcap_{a\in A}\ker a$.\end{claim}

Note that every $a\in A$ can be represented as
 $$n_1f_{Y_1}+n_2f_{Y_2}+...+n_mf_{Y_m}$$
 for $m\in \mathbb{N}$, $n_1,n_2,...,n_m\in \mathbb{Z}$ and $Y_1, Y_2,..., Y_m\in [X]^{<\tau}$.
Then $\ker a$ contains $$\bigcap_{i=1}^m \ker f_{Y_i}=\bigcap_{i=1}^mF(X\setminus Y_i)=F(X\setminus\bigcup_{i=1}^m Y_i).$$
Note that $Y_a:=\bigcup_{i=1}^m Y_i$ has cardinality $<\tau$ and $\tau$ is regular.
So the cardinality of $Y:=\bigcup_{a\in A}Y_a$ is again less than $\tau$.
Then $f_Y\in G^*$ and $F(X\setminus Y)=\ker f_Y$ is a closed $G_\delta$ subgroup of $G$.
It is evident that $F(X\setminus Y)\subseteq F(X\setminus Y_a)=\ker a$, for any $a\in A$.
So the claim has been proven.

For any infinite $\lambda\leq \tau$, and a $G_\lambda$ set $U$ of $G$ such that $0\in U$, there exists a subset $A$ of $G^*$ such that $|A|=\lambda$ such that $\bigcap_{a\in A}\ker a\subseteq U$.
By Claim 1, we obtain that there exists $Y\in [X]^{<\tau}$ with $F(X\setminus Y)\subseteq U$.
Since there exists an element $f_Y\in G^*$ having kernel $F(X\setminus Y)$, we know that $F(X\setminus Y)$ is a closed $G_\delta$ subgroup.
So $G$ is pseudo-$\lambda$-fine.
To see that $G$ is not $\tau$-fine, we only need to prove that $\psi(G)=\tau$.
While it follows from immediately from Claim 1 that the intersection of less than $\tau$ many open neighbourhood of $1_G$ always contains a nontrivial subgroup.
So the proof is complete.
\end{proof}

\begin{corollary}
There exists a pseudo-$\mathfrak{c}$-fine precompact abelian group whose pseudocharacter is uncountable.
\end{corollary}

Due to Corollary \ref{Coro} and Theorem \ref{pfi}, we will conclude this subsection with the following question.

\begin{question}
Let $\tau$ be a regular cardinal such that $\mathfrak{c}^+<\tau\leq (2^\mathfrak{c})^+$. Is there a precompact abelian group whose pseudo-fineness index is $\tau$?
\end{question}

\subsection{Pseudocompactness Destroys Pseudo-$\tau$-fineness}
In this subsection we will see that if a topological group is pseudocompact and pseudo-$\tau$-fine for some uncountable $\tau$, then it must be metrizable and compact.
We begin with an easy lemma.

\begin{lemma}\label{sub}
Let $G$ be a group and $H$ a normal subgroup.
We denote by $\pi$ the quotient homomorphism $G \to G/H$.
Then for two normal subgroups $A, B$ of $G$, we have $A = B$ whenever $A \cap H = B \cap H$ and $\pi(A) = \pi(B)$ hold.
\end{lemma}

\begin{proof}
Let $C = A \cap H = B \cap H$ and $\varphi: G \to G/C$ the quotient homomorphism.
Then there exists a natural homomorphism $\pi': G/C \to G/H$ with $\pi = \pi' \circ \varphi$.
From $\pi(A) = \pi(B)$, it follows that
$$ \varphi(A) = \pi'^{-1}(\pi(A)) = \pi'^{-1}(\pi(B)) = \varphi(B). $$
So we have $A = \varphi^{-1}(\varphi(A)) = \varphi^{-1}(\varphi(B)) = B$.
\end{proof}

Let $G$ be a topological group. A closed normal subgroup $N$ of $G$ is called a {\em co-Lie} subgroup of $G$ if $G/N$ is a Lie group. The cardinality of the set of all co-Lie subgroups of $G$ is called the {\em co-Lie character} of $G$ and denoted by $cL(G)$.
This cardinal function will play a central role in the proof of our final result.

\begin{lemma}\label{colie}
Let $G$ be a topological group and $H$ a compact normal subgroup. Then $cL(G)\leq cL(H) \cdot cL(G/H)$.
\end{lemma}
\begin{proof}
Let $\pi: G \to G/H$ be the quotient homomorphism and $N$ a co-Lie subgroup of $G$.
Since $H$ is compact, $H/(N\cap H)$ is topologically isomorphic to the closed subgroup $HN/N$  of the Lie group $G/N$.
Hence $H/(N\cap H)$ is also a Lie group and $N\cap H$ is a co-Lie subgroup of $H$.
Moreover, it follows from the topological isomorphism
$\pi(G)/\pi(N)\cong G/N$
that $\pi(N)$ is a co-Lie subgroup of $\pi(G)=G/H$.

By Lemma \ref{sub}, $N$ is totally determined by the
pair $(N \cap H, \pi(N))$, where the first is a co-Lie subgroup of $H$ and the second is a co-Lie subgroup $G/H$. So we have $cL(G) \leq cL(H) \cdot cL(G/H)$.
\end{proof}

\begin{lemma}\label{cl=w}
For any infinite pseudocompact group $G$ we have that $cL(G)\cdot \omega = w(G)$.
\end{lemma}
\begin{proof}
Let $K$ be the completion of $G$.
Then we have that $w(K) = w(G)$.
Moreover, if $N$ is a co-Lie subgroup of $G$, then the closure of $N$ is a co-Lie subgroup of $K$.
And conversely, if $M$ is a co-Lie subgroup of $K$, then by Fact \ref{pseudo}, the restriction of the canonical homomorphism $K \to K/M$ to $G$ remains open and surjective.
So $M \cap G$ is a co-Lie subgroup of $G$. Moreover, $M \cap N$ is dense in $M$ because $M$ is a $G_\delta$ subgroup and $G$ is $G_\delta$ dense in $K$.
Then $M \to M \cap G$ gives the one-to-one correspondence between the sets of co-Lie subgroups of $K$ and of $G$.
So we have that $cL(G) = cL(K)$.
From the above argument, we may assume without loss of generality that $G = K$, i.e., $G$ is compact.

If $cL(G)$ is finite, then $G$ is evidently a Lie group. So $w(G)=\omega=cL(G)\cdot \omega$ and we are done.

From now on we assume that $cL(G)$ is infinite and we will see that $cL(G)=w(G)$.
Let $\mathcal{C}$ be the set of all co-Lie subgroups of $G$.
It is well known that $G$ embeds into $P := \prod_{N \in \mathcal{C}} G/N$ as a topological subgroup.
So we have that
$$ w(G) \leq w(P) = |\mathcal{C}| = cL(G). $$

It remains to show that $cL(G) \leq w(G)$. 
We now consider three cases, which will be enough to prove the lemma.
Let $\tau$ be the weight of $G$.

\begin{case}\label{caseabel}
$G$ is abelian.
\end{case}
Let $D$ be the Pontryagin dual group of $G$.
Since compact abelian Lie groups are of the form $\mathbb{T}^n \times F$, where $n \in \mathbb{N}$ and $F$ is finite, a closed subgroup $H$ of $G$ is co-Lie if and only if the Pontryagin dual group of $G/H$ is a finitely generated subgroup of $D$.
So $cL(G)$ coincides with the number of finitely generated subgroups of $D$.
Let $[D]^{<\omega}$ be the set of finite subsets of $D$, then we have
$$ cL(G) \leq |[D]^{<\omega}| = |D| = w(G). $$

\begin{case}\label{casecon}
$G$ is connected.
\end{case}
Let $C$ be the center of $G$. Then, by the structure theorem of connected compact groups \cite[Theorem 9.24]{HM}, $G/C$ is a product of at most $\tau$ many non-trivial simple compact Lie groups.
Here, ``simple'' means ``topologically simple'', so a simple Lie group is a center-free Lie group with a simple Lie algebra.
Let us then denote $G/C$ by $H = \prod_{i \in I} S_i$, where $\kappa := |I| \leq \tau$ and each $S_i$ is simple.
Moreover, the same theorem also asserts that a closed subgroup of $H$ is indeed a subproduct.
So a co-Lie subgroup of $H$ is of the form $\prod_{j \in J} S_j$ with $J$ a co-finite subset of $I$.
If $\kappa$ is infinite, then there are exactly $\kappa$ many co-finite subsets of $I$,
we have that $cL(G/C) = \kappa \leq \tau$.
And if $\kappa$ is finite, then $cL(G/C) < \omega \leq \tau$.

Now consider the abelian compact group $C$.
It is evident that if $C$ is infinite, then 
$$ cL(C) \leq w(C) \leq \tau $$ 
by Case \ref{caseabel}.
And if $C$ is finite, then $cL(C)$ is finite so less than $\tau$ as well.
In summary, we have that $cL(C) \leq \tau$ and in view of Lemma \ref{colie} we have 
$$ cL(G) \leq cL(C) \cdot cL(G/C) \leq \tau \cdot \tau = \tau. $$

\begin{case}\label{casetd}
$G$ is totally disconnected, i.e., $G$ is profinite.
\end{case}
In this case, a normal subgroup of $G$ is co-Lie if and only if it is open.
As is known, an infinite profinite group of weight $\tau$ has exactly $\tau$ many open subgroups \cite{RZ}, so we have that $cL(G) = \tau$.

Now let us go back to the general case.
Let $G_0$ be the identity component of $G$.
Then $G/G_0$ is profinite.
If $cL(G_0)$ is finite, then $cL(G/G_0)$ is infinite according to Lemma \ref{colie} and our assumption that $cL(G)$ is infinite.
Now, by Case \ref{casetd},
 $$cL(G)=cL(G/G_0)\leq w(G/G_0)=w(G/G_0)\leq w(G).$$
 Similarly, if $cL(G/G_0)$ is finite, i.e., $G_0$ is open in $G$, then $cL(G_0)$ is infinite.
 In view of Case \ref{casecon}, we have
 $$cL(G)=cL(G_0)\leq w(G_0)=w(G).$$
It remains to consider the case that both $cL(G_0)$ and $cL(G/G_0)$ are infinite.
The above two Cases \ref{casecon} and \ref{casetd} imply that $cL(G_0) \leq w(G_0)$ and $cL(G/G_0) \leq w(G/G_0)$.
By applying Lemma \ref{colie} again, we have 
$$ cL(G) \leq cL(G_0) \cdot cL(G/G_0) \leq w(G_0) \cdot w(G/G_0) \leq w(G). $$
Now the proof is complete.
\end{proof}

\begin{remark}
\begin{itemize}
\item[(a)] It is evident that if $cL(G)$ is infinite then we will obtain the equality $cL(G)=w(G)$ in the above lemma.
\item[(b)] It is known that for any infinite topological group $G$ with a closed normal subgroup $H$, we have that $w(G)=w(H)\cdot w(G/H)$.
So combining Lemma \ref{colie} and item (a) we get the equality $cL(G)=cL(H)\cdot cL(G/H)$ when $G$ is a compact group with infinite co-Lie character.
\end{itemize}
\end{remark}

\begin{lemma}
Let $G$ be a pseudocompact group with uncountable pseudocharacter and $\tau$ be an uncountable cardinal. Then $G$ is pseudo-$\tau$-fine if and only if for every closed subgroup $N$ of $G$, $\psi(G/N) \leq \tau$ only when $G/N$ is first countable and compact.
\end{lemma}

\begin{proof}
We denote by $K$ the Weil completion of $G$.
Then, $G$ is $G_\delta$ dense in $K$.
Let $N$ be a closed subgroup of $G$ such that $\psi(G/N) = \tau$.
If $G$ were pseudo-$\tau$-fine, then there would be a closed $G_\delta$ subgroup $H$ of $G$ such that $H \leq N$.
Take a sequence $\{U_n : n \in \omega\}$ of identity open neighborhoods in $K$ such that $U_n \subseteq V_n$, $U_n = U_n^{-1}$, and $U_{n+1}U_{n+1} \subseteq U_n$ for any $n \in \omega$.
Then $P := \bigcap_{n \in \omega} U_n$ is a closed $G_\delta$ subgroup of $K$ satisfying that $P' := P \cap G \subseteq H$.
Let $\pi: K \to K/P$.
By $G_\delta$ denseness of $G$ in $K$, we then conclude that the restriction of $\pi$ to $G$ remains open and $\pi(G) = K/P$.
That is saying that $G/P' \cong K/P$.
Since the singleton $\{\pi(1_K)\}$ is also a $G_\delta$ set in $K/P$, we know that the compact space $K/P$ is first countable.
As an open continuous image of $G/P'$, $G/N$ should also be first countable and compact.
\end{proof}

We are now going to prove the final result of this section.

\begin{theorem}
Let $\tau$ be an infinite cardinal. Then a pseudocompact group $G$ cannot be pseudo-$\tau$-fine whenever $G$ is not metrizable.
\end{theorem}
\begin{proof}
Since $G$ is not metrizable, $w(G)$ is uncountable.
By Lemma \ref{cl=w}, $G$ has a family $\{N_\alpha : \alpha < \omega_1\}$ of pairwise distinct co-Lie subgroups.
Let $N = \bigcap_{\alpha < \omega_1} N_\alpha$.
Then $G/N$ admits an injective continuous homomorphism into $\prod_{\alpha < \omega_1} G/N_\alpha$; so $\psi(G/N) \leq \omega_1$.

To see that $G$ is not pseudo-$\tau$-fine, we need only to check that $G$ is not pseudo-$\omega_1$-fine.
And in view of the above lemma, it suffices to prove that $G/N$ is not first countable.
Indeed, since those $N_\alpha$ are pairwise disjoint and $N \subseteq N_\alpha$ for any $\alpha < \omega_1$, we have that $\{N_\alpha / N : \alpha \in \omega_1\}$ are pairwise disjoint co-Lie subgroups of $G/N$.
So according to Lemma \ref{cl=w}, $\chi(G/N) = w(G/N)$ is uncountable.
\end{proof}

\section*{Acknowledgements}
The authors would like to acknowledge the support by NSFC grants NO. 12301089 and 12271258, and the Natural Science Foundation of the Jiangsu Higher Education
Institutions of China (Grant NO. 23KJB110017).

\end{document}